\documentclass[11pt,draft]{amsart}
\usepackage{setspace,units}

\newtheorem{theorem}{Theorem}[section]
\newtheorem{lemma}[theorem]{Lemma}

\theoremstyle{definition}

\newtheorem{remark}[theorem]{Remark}

\numberwithin{equation}{section}

\RequirePackage{amsmath}
\RequirePackage{amssymb}
\RequirePackage{amsthm}
 \RequirePackage{epsfig}

\begin{document}

\date{}
\title[Sumsets of convex sets]
{On a theorem of Schoen and Shkredov on sumsets of convex sets}

\author{ Liangpan Li}

\address{Department of Mathematical Sciences,
Loughborough University, LE11 3TU, UK }
 \email{liliangpan@gmail.com}



\date{}

\begin{abstract}
A set of reals $A=\{a_1,\ldots,a_n\}$ labeled in increasing order is
called convex if there exists a continuous strictly convex function
$f$ such that $f(i)=a_i$ for every $i$. Given a convex set $A$, we
prove
\[|A+A|\gg\frac{|A|^{14/9}}{(\log|A|)^{2/9}}.\]
Sumsets of different summands and an application to a
sum-product-type problem are also studied either as remarks or as
theorems.

\end{abstract}

\maketitle


\section{Introduction}

Let $A=\{a_1,\ldots,a_n\}$ be a set of real numbers  labeled in
increasing order. We say that $A$ is \textsf{convex} if there exists
a continuous strictly convex function $f$ such that $f(i)=a_i$ for
every $i$. Hegyv\'{a}ri (\cite{Hegyvari}), confirming a conjecture
of Erd\H{o}s, proved that if $A$ is convex then
\[|A-A|\gg|A|\cdot\frac{\log|A|}{\log\log|A|},\]
where ``$\gg$" is the Vinogradov notation. This result was later
improved by many authors, see for example
\cite{ElekesNathansonRuzsa,Garaev,Garaev2005,Iosevich,Konyagin,Solymosi}
for related results. Recently, Schoen and Shkredov
(\cite{SchoenShkredov2}), combining
 an energy-type equality (\cite{SchoenShkredov1})
\begin{align}\label{prototype}
E_3(A)=\sum_sE\big(A, A\cap(A+s)\big),
\end{align}
a useful set inclusion relation (see e.g. \cite{Katz,Sanders
1,Sanders 2,Schoen,SchoenShkredov1})
\begin{align}
|(A+A)\cap(A+A+s)|\geq|A+(A\cap(A+s))|,
\end{align}
and an application (see Lemma \ref{Main Lemma} below) of the
Szemer\'{e}di-Trotter incidence theorem (see e.g.
\cite{Kaplan,Szekely,SzemerediTrotter}), proved for convex sets the
following best currently known lower bounds:
\begin{align}
|A+A|&\gg\frac{|A|^{14/9}}{(\log|A|)^{2/3}},\label{formula 1.1}\\
|A-A|&\gg\frac{|A|^{8/5}}{(\log|A|)^{2/5}}.\label{best}
\end{align}
We also remark that Solymosi and Szemer\'{e}di obtained a similar
result for convex sets, establishing $|A\pm A|\gg|A|^{1.5+\delta}$
for some universal constant $\delta>0$.

The purpose of this note is twofold. Firstly, we give a slight
improvement of (\ref{formula 1.1}) as follows:
\begin{theorem}\label{main theorem}
Let $A$ be a convex set.  Then
\begin{align}\label{ggggggggg}
|A+A|\gg\frac{|A|^{14/9}}{(\log|A|)^{2/9}}.
\end{align}
\end{theorem}

Secondly, and most importantly, we will address an application of
the Schoen-Shkredov estimate to a sum-product-type problem.
Erd\H{o}s and Szemer\'{e}di (\cite{Erdos}) once conjectured that the
size of either  the sumset or the productset of an arbitrary set of
the reals must be very large, see \cite{Solymosi2008} for the best
currently known result toward this conjecture and related references
therein. Another type of problem than one can attack regarding
sumset and productset is to assume either one is very small, then
prove the other one is very large. Elekes and Ruzsa
(\cite{ElekesRuzsa}, see also \cite{LiShen,Solymosi2008}) proved
that if the sumset of a set is very small, then its productset must
be very large. On the other hand, if  the productset of a set is
very small, say for example $|AA|\leq M|A|$, then the best currently
known lower bound for the size of its sumset (\cite{Elekes}, see
also \cite{ElekesNathansonRuzsa,LiShen,Solymosi2008}) only is
$|A+A|\geq C_M|A|^{3/2}$.

Roughly speaking, we will show that a set with very small
multiplicative doubling is a ``convex" set. Consequently, we can
derive the following improvement.
\begin{theorem}\label{main theorem 2}
Suppose $|AA|\leq M|A|$.  Then
\begin{align*}
|A+A|&\gg_M \frac{|A|^{14/9}}{(\log|A|)^{2/9}},\\
|A-A|&\gg_M \frac{|A|^{8/5}}{(\log|A|)^{2/5}}.
\end{align*}
\end{theorem}

We remark that one can find direct application of Theorem \ref{main
theorem 2} to the main result in \cite{Li}, in which multi-fold sums
from a set with  very small multiplicative doubling are studied. See
also \cite{BourgainChang,Chang,CrootHart} for some related
discussions on multi-fold sumsets.




We collect some notations used throughout this note.  Denote by
$\delta_{A, B}(s)$ the number of representations of $s$ in the form
$a-b$, $a\in A$, $b\in B$. If $A=B$ we write
$\delta_{A}(s)=\delta_{A, A}(s)$ for simplicity. Furthermore, put
$$E(A,B)=\sum_s\delta_A(s)\delta_B(s)=\sum_s\delta_{A, B}(s)^2$$
and $$E_k(A)=\sum_s\delta_A(s)^k.$$ Let $A_s=A\cap(A+s)$. All
logarithms are to base 2. \textsf{All sets are finite subsets of
real numbers}.

\section{Convexity and energy estimates}

\begin{lemma}[\cite{SchoenShkredov2}]\label{Main Lemma}
Let $A$ be a convex set. Then for any set $B$ and any $\tau\geq1$ we
have
\[\big|\{x\in A-B:\delta_{A, B}(x)\geq\tau\}\big|\ll\frac{|A|\cdot|B|^2}{\tau^3}.\]
\end{lemma}

A special case of Lemma \ref{Main Lemma} for $B=-A$ was established
in \cite{Iosevich}. As applications, we have the following two
lemmas.

\begin{lemma}[\cite{SchoenShkredov2}]\label{lemma 2.2}
Let $A$ be a convex set. Then $E_3(A)\ll|A|^3\cdot\log|A|$.
\end{lemma}

\begin{lemma}\label{corollary 2.3}
Let $A$ be a convex set. Then for any set $B$ we have $E(A,
B)\ll|A|\cdot|B|^{1.5}$.
\end{lemma}

\begin{proof}
Let $\triangle\doteq\frac{E(A, B)}{2|A||B|}$ and we divide $E(A,B)$
into two parts, one is
\[\sum_{s:\delta_{A, B}(s)<\triangle}\delta_{A, B}(s)^2,\]
which is obviously less than half of $E(A, B)$, thus results in the
other part
\[\sum_{s:\delta_{A, B}(s)\geq\triangle}\delta_{A, B}(s)^2,\]
being bigger than  half of $E(A, B)$. Therefore, by Lemma \ref{Main
Lemma} and a dyadic argument,
\[\frac{E(A, B)}{2}\leq\sum_{s:\delta_{A, B}(s)\geq\triangle}\delta_{A, B}(s)^2\ll\sum_{j\geq1}\triangle^2\cdot2^{2j}\cdot
\frac{|A|\cdot|B|^2}{\triangle^3\cdot2^{3j}}\leq\frac{|A|\cdot|B|^2}{\triangle}.\]
This finishes the proof.
\end{proof}


\begin{lemma}\label{lemma 24} Let $A,B$ be any sets. Then
\[\sum_sE(A_s,B)\leq E_3(A)^{2/3}\cdot E_3(B)^{1/3}.\]
\end{lemma}

\begin{proof}
Note  $\delta_{A_s}(t)=\delta_{A_t}(s)$, which in common is
$|A\cap(A+s)\cap(A+t)\cap(A+s+t)|$. Thus
\begin{align*}
\sum_sE(A_s, B)&=\sum_s\sum_t\delta_{A_s}(t)\delta_B(t)=\sum_s\sum_t\delta_{A_t}(s)\delta_B(t)\\
&=\sum_t\sum_s\delta_{A_t}(s)\delta_B(t)=\sum_t\delta_A(t)^2\delta_B(t)\\
&\leq\Big(\sum_t\delta_A(t)^3\Big)^{2/3}\cdot\Big(\sum_t\delta_B(t)^3\Big)^{1/3}=E_3(A)^{2/3}\cdot
E_3(B)^{1/3}.
\end{align*}
This finishes the proof.
\end{proof}

\begin{lemma}\label{lemma 25}
Let $A,B$ be any sets. Then
\[E_{1.5}(A)^2\cdot|B|^2\leq \big(\sum_sE(A_s, B)\big)\cdot E(A,A+B).\]
\end{lemma}

\begin{proof}
By the Cauchy-Schwarz inequality, \[|A_s|^{1.5}\cdot|B|\leq
E(A_s,B)^{1/2}\cdot|A_s+B|^{1/2}\cdot|A_s|^{1/2}.\] First summing
over all $s\in A-A$, then applying Cauchy-Schwarz again gives
\begin{align*}
E_{1.5}(A)^2\cdot|B|^2&\leq\big(\sum_sE(A_s,
B)\big)\cdot\big(\sum_s|A_s+B|\cdot|A_s|\big)\\
&\leq \big(\sum_sE(A_s,
B)\big)\cdot\big(\sum_s|(A+B)_s|\cdot|A_s|\big)\\
&=\big(\sum_sE(A_s, B)\big)\cdot E(A,A+B),
\end{align*}
where the second inequality is due to the set inclusion relation
$A_s+B\subset (A+B)_s$.  This finishes the proof.
\end{proof}

\section{Proof of Theorem \ref{main theorem}}

This section is mainly devoted to the proof of Theorem \ref{main
theorem}. We first claim
\[
E_2(A)^3\ll |A|^3\cdot E_{1.5}(A)^2,
\]
which follows simply from (see also the proof of Lemma
\ref{corollary 2.3})
\[E_2(A)=\sum_{s:\delta_{A}(s)<\triangle}\delta_{A}(s)^2+
\sum_{s:\delta_{A}(s)\geq\triangle}\delta_{A}(s)^2\ll
\sqrt{\triangle}\cdot E_{1.5}(A)+\frac{|A|^3}{\triangle}.\] Then
applying Lemma \ref{lemma 25} with $B=A$, Lemma \ref{lemma 24} and
Lemma \ref{lemma 2.2}, we get
\[\frac{|A|^{12}}{|A+A|^3}\leq E_2(A)^3\ll|A|^3\cdot|A|\cdot(\log|A|)\cdot|A|\cdot|A+A|^{3/2},\]
which is equivalent to
\[
|A+A|\gg\frac{|A|^{14/9}}{(\log|A|)^{2/9}}.
\]
This finishes the proof of Theorem \ref{main theorem}.

\begin{remark}
Let $A,B$ be convex sets.  We remark that one can establish
\begin{align}|A\pm B|^{9}\gg\frac{|A|^6\cdot|B|^8}{(\log|A|)^{4/3}\cdot(\log|B|)^{8/3}}.\end{align}
To this aim, it suffices to note
\begin{align*}\frac{|A|^2\cdot|B|^2}{|A\pm B|}&\leq E(A,B)=\sum_s\delta_A(s)\cdot\delta_{B}(s)\\&\leq \big(\sum_s\delta_A(s)^{3/2}\big)^{2/3}\cdot
\big(\sum_s\delta_B(s)^{3}\big)^{1/3}\\&=E_{1.5}(A)^{2/3}\cdot
E_3(B)^{1/3},\end{align*} then turning to Lemmas \ref{lemma
2.2}$\sim$\ref{lemma 25} to get the desired inequality.

\end{remark}


\begin{remark} Let $A,B$ be convex sets.
We remark that one can establish
\begin{align}\label{hhhhhhhhhhhhhhhhhhh} |A-A|^{2}\cdot|A\pm
B|^3\gg\frac{|A|^6\cdot|B|^2}{(\log|A|)^{4/3}\cdot(\log|B|)^{2/3}}.
\end{align}
To this aim, it suffices to note from the H\"{o}lder inequality that
\[\frac{|A|^6}{|A-A|}\leq E_{1.5}(A)^{2},\]
then turning to Lemmas \ref{lemma 2.2}$\sim$\ref{lemma 25} to get
the desired inequality.
\end{remark}






\section{Proof of Theorem \ref{main theorem 2}}


 \begin{lemma}

Let $A$ be a set of the form $f(Z)$, where $f$ is  a continuous
strictly convex function,  $|Z+Z|\leq M|Z|$. Then for any set $B$
and  any $\tau\geq1$,
\[\big|\{x\in A-B:\delta_{A, B}(x)\geq\tau\}\ll M^3\cdot\frac{|A|\cdot|B|^2}{\tau^3}.\]

\end{lemma}

\begin{proof}
Without loss of generality, we may assume that $f$ is monotonically
increasing, and $1\ll\tau\leq\min\{|A|, |B|\}$.
 Let $G(f)$ denote the graph of $f$ in the plane. For any
$(\alpha,\beta)\in\mathbb{R}^2$, put
$L_{\alpha,\beta}=G(f)+(\alpha,-\beta).$ Define the pseudo-line
system $\mathcal L=\{L_{z,b}: (z,b)\in Z\times B\}$, and the set of
points $\mathcal P=(Z+Z)\times(A-B).$ By convexity, $|\mathcal
L|=|Z|\cdot|B|=|A|\cdot|B|$.
 Let ${\mathcal P}_{\tau}$ be the set of
points of $\mathcal P$ belonging to at least $\tau$ curves from
$\mathcal L$. By the Szemer\'{e}di-Trotter incidence theorem,
\[\tau\cdot|{\mathcal P}_{\tau}|\ll(|{\mathcal P}_{\tau}|\cdot|Z|\cdot|B|)^{2/3}+|Z|\cdot|B|+|{\mathcal P}_{\tau}|,\]
from which we can deduce (see also \cite{SchoenShkredov2})
\[|{\mathcal P}_{\tau}|\ll\frac{|Z|^2\cdot|B|^2}{\tau^3}.\]

Next, suppose $\delta_{A, B}(x)\geq\tau$. There exist
 $\tau$ distinct elements $\{z_i\}_{i=1}^{\tau}$ from $Z$, $\tau$
distinct elements $\{b_i\}_{i=1}^{\tau}$ from $B$, such that
$x=f(z_i)-b_i\ (\forall i)$. Now we define $Z_i\triangleq z_i+Z \
(\forall i)$ and ${\mathcal
M}_x(s)\triangleq\sum_{i=1}^{\tau}\chi_{Z_i}(s)$, where
$\chi_{Z_i}(\cdot)$ is the characteristic function of $Z_i$. Since
\[(z_i+z,x)=\big(z_i,f(z_i)\big)+(z,-b_i)\in L_{z,b_i}\ \ (\forall z, \forall i),\]
we have $(s,x)\in {\mathcal P}_{{\mathcal M}_x(s)}$. Obviously,
\[\sum_{s\in Z+Z}{\mathcal M}_x(s)=\sum_{i=1}^{\tau}\sum_{s\in Z+Z}\chi_{Z_i}(s)\geq\tau|Z|.\]
Thus by the standard popularity argument,
\[\big|\{s\in Z+Z: {\mathcal M}_x(s)\geq\frac{\tau}{2M}\}\big|\geq\frac{|Z|}{2}.\]
This naturally implies
\[\big|\{x\in A-B:\delta_{A, B}(x)\geq\tau\}\big|\cdot\frac{|Z|}{2}\leq|{\mathcal P}_{\frac{\tau}{2M}}|,\]
and consequently,
\[\big|\{x\in A-B:\delta_{A, B}(x)\geq\tau\}\big|\ll\frac{|{\mathcal P}_{\frac{\tau}{2M}}|}{|Z|}\ll M^3\cdot\frac{|Z|\cdot|B|^2}{\tau^3}
=M^3\cdot\frac{|A|\cdot|B|^2}{\tau^3}.\] This finishes the proof.
\end{proof}

It is rather easy to observe that, any property holds for convex
sets in this note should also hold for sets of the form $f(Z)$,
where $f$ is a continuous strictly convex function,  $|Z+Z|\leq
M|Z|$, with $\gg$ replaced by $\gg_M$.

As applications, let $A$ be a finite set of positive real numbers
with $|AA|\leq M|A|$. Then $A=\exp(Z)$, $Z=\ln A$, $|Z+Z|=|AA|\leq
M|A|=M|Z|$. Consequently, (\ref{ggggggggg}) and
(\ref{hhhhhhhhhhhhhhhhhhh}) hold for such an $A$. This suffices to
prove  Theorem \ref{main theorem 2}. We are done.

\textbf{Acknowledgements.} This work was supported by the NSF of
China (11001174).

\end{document}